\documentclass[12pt]{article}

\usepackage[utf8x]{inputenc}
\usepackage[english]{babel}
\usepackage[T1]{fontenc}
\usepackage{amsmath,amscd}
\usepackage{amsfonts}
\usepackage{amssymb}
\usepackage{amsthm}
\usepackage{mathrsfs}
\usepackage{graphicx}
\usepackage[all,cmtip]{xy}
\usepackage{textcomp}
\usepackage{url}
\usepackage{amsopn,latexsym,amscd}
\usepackage{color}
\usepackage{todonotes}
\usepackage{fullpage}
\usepackage{amsmath,graphicx}

\newtheorem{teo}{Theorem}[section]

\newtheorem{rem}{Remark}[section]
\newtheorem{prop}{Proposition}[section]
\newtheorem{cor}{Corollary}[section]
\newtheorem{es}{Example}[section]
\newtheorem{defin}{Definition}[section]

\newcommand{\Z}{{\mathbb{Z}}}
\newcommand{\C}{{\mathbb{C}}}
\newcommand{\R}{{\mathbb{R}}}

\newcommand{\N}{{\mathbb{N}}}

\newcommand{\G}{{\mathcal{G}}}

\newcommand{\MS}{{\mathcal{S}}}
\newcommand{\Proj}{{\mathbb{P}}}
\newcommand{\A}{{\mathcal{A}}}
\newcommand{\Pc}{{\mathcal{P}}}

\newcommand{\K}{{\mathbb{K}}}
\newcommand{\Sc}{{\mathcal{S}}}
\newcommand{\Cc}{{\mathcal{C}}}
\newcommand{\Bc}{{\mathcal{B}}}
\newcommand{\Fc}{{\mathcal{F}}}

\newcommand{\Hc}{{\mathcal{H}}}

\newcommand{\emme}{{\mathcal M}}

\newcommand{\supp}{{\text{supp}}}

\title{Exponential formulas for models of complex reflection groups}
\author{Giovanni Gaiffi\footnote{ Dipartimento di Matematica, Universit\`a di Pisa, 
Largo Bruno Pontecorvo, 5, 56127, Pisa, \texttt{gaiffi@dm.unipi.it}}}
\date{\today}

\begin{document}
\maketitle
\begin{abstract}
In this paper we  find  exponential formulas for the Betti numbers of the  De  Concini-Procesi  minimal  wonderful models \(Y_{G(r,p,n)}\) associated to the complex reflection groups \(G(r,p,n)\). Our formulas are  different from the ones  already known in the literature:  they are obtained by  a new combinatorial encoding  of  the elements of a basis of the cohomology by means of set partitions with weights and exponents.

We also point out that a similar   combinatorial encoding  can be used to describe the faces of the real spherical wonderful models of type \(A_{n-1}\)(\(=G(1,1,n)\)), \(B_n\) (\(=G(2,1,n)\)) and \(D_n\)(\(=G(2,2,n)\)). This  provides exponential formulas for the  \(f\)-vectors of  the associated nestohedra: the Stasheff's associahedra (in this case closed formulas are well known)  and  the graph associahedra of type  \(D_n\). 
\end{abstract}

\section{Introduction}

Let us start by fixing some  notations.
 First we recall that the finite irreducible complex reflection groups, according to the Shephard-Todd classification (see \cite{shephardtodd}), are the groups \(G(r,p,n)\), with \(r,p,n\in \Z^+\) and \(p|r\), plus 34 exceptional groups.

Let \(C(r)\) be the cyclic group of  order \(r\) generated by a primitive \(r\)-th root of unity \(\zeta\). The group \(G(r,1,n)\),  the full monomial group, is the wreath product of \(C(r)\) and  the symmetric group \(S_n\). It can also be described as the group generated by all the complex reflections in \(GL(\C^n)\) whose reflecting hyperplanes are the hyperplanes with equations
\(x_i=\zeta^{\alpha}x_j\), where \(\alpha=0,...,r-1\), and  \(x_i=0\).

Its elements are   all the linear transformations \(g(\sigma, \epsilon)\: : \:  \C^n\rightarrow \C^n\) defined on the standard basis by
\[g(\sigma,\epsilon)e_i=\epsilon(i)e_{\sigma(i)}\]
where \(\sigma\in S_n\) and \(\epsilon\) ranges among the functions from \(\{1,...,n\}\) to \(C(r)\).

The group \(G(r,p,n)\) is the subgroup of \(G(r,1,n)\) consisting of all the  \(g(\sigma, \epsilon)\) such that the product \(\epsilon(1)\epsilon(2)\cdots \epsilon(n)\) is a power of \(\zeta^p\).
If \(p<r\) the sets of reflecting hyperplanes of \(G(r,p,n)\) and \(G(r,1,n)\) coincide, and  
their intersection lattice is the Dowling lattice \(Q_n(\Z^r)\) (see \cite{dowling}). The set of reflecting hyperplanes of \(G(r,r,n)\) is obtained from that of \(G(r,1,n)\) by deleting the coordinate hyperplanes \(x_i=0\).

As important  examples, we observe that \(G(1,1,n)=S_n\) is the Weyl group of type \(A_{n-1}\), while \(G(2,1,n)\) and \(G(2,2,n)\) are respectively the Weyl groups of type \(B_n\) and \(D_n\).

\subsection{The interest of the models \(Y_{G(r,p,n)}\)}

Wonderful models have been constructed by De Concini-Procesi in their seminal papers \cite{DCP2} and \cite{DCP1}. They play a relevant role in several   fields: subspace and toric  arrangements (see  \cite{DCP3}, \cite{feichtneryuz}),   
configuration spaces, box splines and  index theory (see the exposition in  \cite{DCP4}),  tropical geometry (see for instance \cite{feichtnersturmfels} and  the survey \cite{denham}) and discrete  geometry    (see \cite{feichtner} for further references).  
We will recall  in Sections \ref{subsecbuilding} and \ref{subsecwonderful} the construction of these models, including the definitions of nested sets and building sets,  and their main properties. The importance of the models associated with reflection groups, i.e. with the hyperplane arrangements  given by their reflecting hyperplanes,  was at first  pointed out  by the example of type \(A\): the minimal projective De Concini-Procesi model  of type \(A_{n-1}\) is isomorphic to the  moduli space \({\overline M_{0,n+1}}\) of (\(n+1\))-pointed stable  curves of genus 0.
This isomorphism  carries on the cohomology of the models of  type \(A_{n-1}\) an `hidden' extended  action of \(S_{n+1}\) that has been  studied by several authors (see for instance \cite{getzler},  \cite{rains2009}, \cite{etihenkamrai}).

Also the other models \(Y_{G(r,p,n)}\) appeared in the literature in several contexts.
They  are crucial objects in representation theory, since they provide natural geometric representations of \(G(r,p,n)\). They were studied from this point of view    by Henderson in  \cite{hendersonwreath}, where    recursive character formulas  for the action of \(G(r,1,n)\) and \(G(r,r,n)\) on their cohomology were described, as well as their  specializations that   give recursive formulas for the Betti numbers. We note here that the model \(Y_{G(r,1,n)}\)  is equal to \(Y_{G(r,p,n)}\)  if  \(p<r\), since the underlying reflection arrangement is the same.
We recall that recursive formulas for the Betti numbers in the cases \(A_n\), \(B_n\) and \(D_n\) have been obtained  also in \cite{YuzBasi} and   \cite{GaiffiBlowups} (for the  \(A_n\) case these formulas have been found in several other papers devoted to  the moduli spaces approach, see for instance \cite{Manin}).

We remark that in  the case of a finite real  reflection group \(G\),  one  can construct  a complex minimal model \(Y_G\) and  also a minimal real  compact  model \(\overline {Y}_G\):  formulas for the action on the cohomology of \(\overline {Y}_G\)  in the \(A_n\) case  appear in \cite{rains2009} while in \cite{hendersonrains}  the cases of the other finite Coxeter groups  are dealt with.

The combinatorial and discrete geometric interest of  the real models \(\overline {Y}_G\)  comes from the observation that they can  be obtained by glueing some nestohedra: for instance, the models   \(\overline {Y}_{G(1,1,n)}\)  and \(\overline {Y}_{G(2,1,n)}\)  are obtained by glueing Stasheff's associahedra, while   the models \(\overline {Y}_{G(2,2,n)}\) are obtained by glueing graph associahedra of type \(D_n\) (in the sense of Carr and Devadoss, see \cite{carrdevadoss}). There are also  non minimal  De Concini-Procesi models (see \cite{GaiffiServenti2} for a classification), whose construction involves the glueing of permutohedra and other  nestohedra (see \cite{gaiffipermutonestoedra}).


Finally we would like to mention  that the  minimal complex model \(Y_G\), when \(G\) is an irreducible  finite complex reflection group, plays a  role in the theory of braid groups: for instance, the elements in  the center of the pure braid group $PB_G$ (resp. the braid group \(B_G\)) associated to $G$  are easily described in terms of the geometry of  \(Y_G\) (resp. \(Y_G/G\)),  as well as the elements in the center of the parabolic subgroups of \(PB_G\) (resp. \(B_G\), see \cite{callegarogaiffilochak}).

\subsection{A combinatorial approach}
In \cite{callegarogaiffi3}  a new  exponential  (non recursive)  formula for the Betti numbers of the models \(Y_{G(1,1,n)}\) has been found,  using the following  combinatorial approach.
In \cite{YuzBasi} (see also \cite{GaiffiBlowups}) a monomial  basis of  \(H^*(Y_{G(1,1,n)})\) was  described;  the elements of this basis can be represented by graphs, that are some oriented rooted trees  on \(n\) leaves, with exponents attached to the internal vertices. 
Now let us focus on the trees that have \(k\) internal vertices;  in \cite{gaifficayleynumbers}  a bijection between these trees and the partitions of \(\{1,...,n+k-1\}\) into \(k\) parts of cardinality \(\geq 2\) has been described (this  is in fact a variant of a bijection shown in \cite{PeErdos}).
It turns out that, using this bijection, a new representation of the monomials of the basis  of  \(H^*(Y_{G(1,1,n)})\) is provided by {\em partitions with exponents}. The generating function for these  partitions  is expressed by an exponential series (see Theorem \ref{teo:formulapoicareminimale} in Section \ref{classicbraid}, where we  recall  the results on this type A   case).

In this paper we  extend  to  all the groups  \(G(r,p,n)\) the above described combinatorial approach. Here  the combinatorial structure (described in Section \ref{sec:extension})  is richer: vertices of two types appear ({\em strong} and {\em weak} vertices) as well as weights attached to the vertices and the leaves (in addition to exponents attached to the internal vertices, as in the \(A_n\) case).
 Even if the  combinatorial picture is more complicated, we obtain also in this more general case  exponential formulas for the generating functions of the Betti numbers. These formulas are the content of Theorems \ref{teog(r1n)} and \ref{teog(rrn)} in Section \ref{section:seriesgr1n}.

In the  last two sections of the paper  we show  an application of the same principles (the encoding of the combinatorics of nested sets by weighted partitions) to the counting of  the faces of some polytopes associated to the real reflection groups \(G(1,1,n)\), \(G(2,1,n)\) and \(G(2,2,n)\).

We start by recalling, in Section \ref{sec:spherical}, the construction of the minimal spherical model $CY_{G}$ associated with a real reflection group \(G\). This is a smooth manifold with corners and it is diffeomorphic to a 
 disjoint union of polytopes that belong to the family of nestohedra. In   \cite{gaiffipermutonestoedra} a linear realization of  $CY_{G}$ is provided: the polytopes involved lie inside  the chambers of the reflection arrangement;  in every chamber we find a copy of  the graph associahedron  \(P_G\) of type \(G\) (i.e. the graph polytope defined in  \cite{devadoss} and \cite{carrdevadoss} associated with the Dynkin diagram  of type \(G\)).
 We show in Section \ref{Eulercomputation} that the faces of the polytopes  appearing in $CY_{G}$ can be indexed by weighted  {\em internally ordered} partitions, i.e. the parts of the partitions are ordered sets. This gives rise to exponential formulas for the generating series of the \(f\)-vectors of the polytopes \(P_G\) when \(G=G(1,1,n), G(2,1,n), G(2,2,n)\).
 
 In the first two cases the components of the \(f\) vectors are the well known Kirkman-Cayley numbers (a  closed  formula for these numbers dates back to  Cayley's paper   \cite{Cayley}), since the associated polytopes are Stasheff's associahedra.  In the \(G(2,2,n)\) case (see Theorem \ref{teo:politopodn}) our  formulas  generate   the \(f\)-vectors of the graph polytopes of type \(D_n\).
 As a final remark, we notice that specializing these generating series  we can obtain   formulas for the generating series of  the Euler characteristic of the corresponding real compact  De Concini-Procesi models, that can be compared with the closed  formulas described in \cite{hendersonrains}.

\section{Models}
In this section we will recall the basic facts about De Concini-Procesi models of subspace arrangements, introduced in the seminal papers \cite{DCP2}, \cite{DCP1}.

\subsection{Building sets and nested sets}
\label{subsecbuilding}
Let $V$ be a finite dimensional vector space over a field \(\K\)
and let  $\G$ be a finite set of subspaces of the dual space $V^*$.  We denote  by $\Cc_\G$ its closure under the sum.

\begin{defin}
 Given a subspace $U\in\Cc_\G$, a \textbf{decomposition of} $U$ in $\mathbf{\Cc_\G}$ is a collection
$\{U_1,\cdots,U_k\}$ ($k>1$) of non zero subspaces in $\Cc_\G$ such that
\begin{enumerate}
 \item $U=U_1\oplus\cdots\oplus U_k$
 \item for every subspace $A\subset U$, $A\in\Cc_\G$, we have $A\cap U_1,\cdots,A\cap U_k \in \Cc_\G$ and
$A=\left(A\cap U_1\right)\oplus\cdots\oplus \left(A\cap U_k\right)$.
\end{enumerate}
\end{defin}
\begin{defin}
\label{def:irreducible}
 A subspace $F\in\Cc_\G$ which does not admit a decomposition  is called \textbf{irreducible} and the set of
irreducible subspaces is denoted by $\mathbf{\Fc_\G}$.
\end{defin}
One can prove  that 
every subspace $U\in\Cc_\G$ has a unique decomposition into irreducible subspaces.
The set $\mathbf{\Fc_\G}$ of the irreducible spaces and the set $\mathbf{\Cc_\G}$ are  {\em building sets} in the sense of the following definition:
\begin{defin}\label{building}
 A collection $\G$ of subspaces of $V^*$ is called \textbf{building} if every element $C\in\Cc_\G$ is the direct sum
$G_1\oplus\cdots\oplus G_k$ of the set of maximal elements $G_1,\cdots,G_k$ of $\G$ contained in $C$.
\end{defin}


\begin{defin}(see \cite{DCP1}, Section 2.4) \label{Gnested}
 Let $\G$ be a building set of subspaces of $V^*$. A subset $\Sc\subset \G$ is called $\mathbf{\G}$\textbf{-nested} if and only if 
 for every  subset $\{A_1,\cdots,A_k\}$ (\(k\geq 2\)) of pairwise non comparable elements of $\Sc$  the subspace   $A=A_1+\cdots +
A_k$ does not belong to  $\G$.

\end{defin}

After De Concini and Procesi's papers \cite{DCP2} and  \cite{DCP1},  building sets and nested sets turned out to play a relevant role in combinatorics.  For instance, in \cite{feichtnerkozlovincidence}  building sets and nested sets  were defined in the more general context of meet-semilattices and they also appeared in connection with special polytopes, called nestohedra (see   \cite{postnikov},
\cite{postnikoreinewilli}, \cite{zelevinski}, \cite{petric2}, \cite{gaiffipermutonestoedra}). In Sections \ref{sec:spherical} and \ref{Eulercomputation} we will deal with some of these polytopes.

\subsection{Wonderful models}
\label{subsecwonderful}

Let us take  \(\K=\C\)  as the base field and consider a finite subspace arrangement in the complex vector space $V$. We  describe this arrangement by the dual arrangement \(\G\) in \(V^*\) (for every $A\in\G$, we   denote by  $A^\perp$ its
annihilator in $V$). The complement in \(V\) of the arrangement \(\{A^\perp \: | \: A\in \G\}\) will be denoted by \(\emme(\G)\).\\
For every $A\in\G$ we have a rational map defined outside of \(A^{\perp}\):
$$\pi_A:V\longrightarrow V/A^\perp \longrightarrow \Proj \left( V/A^\perp \right).$$

We then consider the embedding 
$$\phi_\G:\emme(\G)\longrightarrow V\times\prod_{A\in\G}\Proj\left( V/A^\perp \right)$$
given by the  inclusion on the first component and by the maps \(\pi_A\) on the other components. 

\begin{defin}
\label{definizionemodelli}
The De Concini-Procesi model
\(Y_{\G}\) associated to  $\G$ is the closure of $\phi_\G \left(\emme(\G) \right)$ in
$V\times\prod_{A\in\G}\Proj\left(V/A^\perp\right)$.
\end{defin}

These  {\em wonderful models}  are  particularly interesting when the arrangement $\G$ is building: they turn out to be 
smooth varieties and the complement of \(\emme(\G)\) in \(Y_\G\) is a divisor with normal crossings. The irreducible components of this divisor are  indexed by \(\G\): 
if   $p$ is the 
projection of $
Y_{\G}$ onto
	the first component \(  V\),  then for every \(G\in \G\) we denote by  ${ \mathcal  D}_G$ the unique 	 irreducible  component 		 such that
	$ p({ \mathcal  D}_G)=G^\perp$.  
	
	A complete characterization of the boundary is then provided by the 
		observation that, if we consider  a collection \( 
	{ \mathcal T} \) of subspaces in \( \G \),  then  \[ 
	{ \mathcal D}_{{ \mathcal T}}= \bigcap_{A\in { \mathcal 
	T}}{\mathcal D}_{A}\] is non empty if and only if \( {\mathcal T} \) 
	is \(\G\)-nested, and in this case \(	{\mathcal D}_{{\mathcal T}}\)  is a smooth irreducible subvariety obtained as a normal crossing intersection.

A presentation of the   integer cohomology rings of the models \(Y_\G\) was provided  in  \cite{DCP1}. They are torsion free, and in  \cite{YuzBasi}  Yuzvinski explicitly described  $\Z$-bases (see also  \cite{GaiffiBlowups} that  extends this description giving bases for   the cohomology  of  the    components of the boundary).  We  briefly  recall these results.

Let $\G$ be a building set of subspaces of $V^*$.  If $\Hc\subset
\G$ and $B\in\G$ is such that $A\subsetneq B$ for each $A\in\Hc$,  one  defines
$$d_{\Hc,B}:=\dim B - \dim \left(\sum_{A\in\Hc} A\right).$$
In the polynomial ring $\Z[c_A]_{A\in\G}$, we consider the ideal \(I\) generated by the polynomials  
$$P_{\Hc,B}:=\prod_{A\in\Hc}c_A\left(\sum_{C\supset B}c_C\right)^{d_{\Hc,B}}$$
 as $\Hc$ and $B$ vary.
\begin{teo}(see \cite{DCP1}, Section 5.2).\\
There is a surjective ring homomorphism
$$\phi \: : \: \Z[c_A]_{A\in\G}\longrightarrow H^*(Y_\G,\Z)$$
whose  kernel is $I$ and such that  \(\phi(c_A)\in H^2(Y_\G,\Z)\) is the Chern class of the divisor \({ \mathcal  D}_A\).  
\end{teo}
\begin{defin}
\label{def:support}
 Let $\G$ be a building set of subspaces of $V^*$. A function
$f:\G\longrightarrow \N$
is $\mathbf{\G}$\textbf{-admissible} (or simply \textbf{admissible}) if $f=0$ or, if $f\neq 0$, the following two conditions hold:
\begin{itemize}
\item $\supp(f)$ is
$\G$-nested
\item for all $A\in\supp(f)$ one has
$f(A)< d_{\supp(f)_A,A}$
\end{itemize}
where $\supp(f)_A:=\{C\in\supp(f):C\subsetneq A\}$.

 A monomial $m_f=\prod_{A\in\G}c_A^{f(A)}\in\Z[c_A]_{A\in\G}$ is \textbf{admissible} if $f$ is admissible.
\end{defin}
\begin{teo}\label{base coomologia}(see  Section 3 of \cite{YuzBasi} and Section 2 of \cite{GaiffiBlowups})\\
 The set $\mathcal{B}_\G$ of all admissible monomials is    a $\Z$-basis of $H^*(Y_\G,\Z)$.
\end{teo}

\section{The braid case}
\label{classicbraid}

Let us consider an hyperplane arrangement  in \(V\), represented by the set \(\Hc\) of the lines in \(V^*\) that are the annihilators of the hyperplanes;   we notice that there is  a minimal  building set that contains \(\Hc\) and it is  the building set of irreducibles \(\Fc_{\Hc}\) (see Definition \ref{def:irreducible}). The  De Concini-Procesi model obtained from \(\Fc_{\Hc}\) is called the minimal De Concini-Procesi model associated to \(\Hc\).

In this section we  will recall some results  in  the case  of the reflection group  \(G(1,1,n)=S_n\). Adopting a notation that will be extended to all the complex reflection groups, we  will  denote by \(\Fc_{G(1,1,n)}\) its associated minimal building set and by \(Y_{G(1,1,n)}\) (instead than \(Y_{\Fc_{G(1,1,n)}}\)) the corresponding minimal model. 

 Let us consider  the real or complexified  braid arrangement, i.e. the  arrangement given in \(\R^n\) or \(\C^n\)  by the hyperplanes  defined by the equations \(x_i-x_j=0\).  The arrangement associated to the  reflection group   \(G(1,1,n)\) coincides with  the  root arrangement of type \(A_{n-1}\) and can be viewed as  the  arrangement in  the quotient space \(V\), where \(V=\R^n/<(1,1,...,1)>\) or \(V=\C^n/<(1,1,...,1)>\), whose hyperplanes  are the projections  of the hyperplanes \(x_i-x_j=0\).  The  projected hyperplanes can still be described by the equations \(x_i-x_j=0\), that are well defined in the quotient.

As we mentioned before, $\Fc_{G(1,1,n)}$ is the minimal building set that contains the  lines in \(V^*\) that are the annihilators of the hyperplanes \(x_i-x_j=0\):  it is made  by all the subspaces in \(V^*\) whose annihilators in \(V\) are described by equations like
 \(x_{i_1}=x_{i_2}= \cdots =x_{i_k}\) (\(k\geq 2\)).
 
 Therefore there  is a bijective correspondence between the elements of $\Fc_{G(1,1,n)}$ and the subsets of $\{1,\cdots,n\}$ of cardinality at
least two:  if the orthogonal of \(A\in \Fc_{G(1,1,n)}\) is the subspace  described by the equation  \(x_{i_1}=x_{i_2}= \cdots =x_{i_k}\) then we represent \(A\) by the set \(\{i_1,i_2,\ldots, i_k\}\).
As a consequence of Definition \ref{Gnested},   a $\Fc_{G(1,1,n)}$-nested set  \(\Sc\) is  represented by a set  (which we still call \(\Sc\)) of subsets of $\{1,\cdots,n\}$ with the property that any of its elements has  cardinality \(\geq 2\) and if \(I\) and \(J\) belong to \(\Sc\) than either \(I\cap J=\emptyset\) or one of the two sets is included into the other.

We observe that we  can  represent a $\Fc_{G(1,1,n)}$-nested set   \(\MS\) 
 by an oriented   forest on \(n\) leaves  in the following way. We consider the set \({\tilde \MS}=\MS\cup\{1\}\cup \{2\}\cup\cdots \cup \{n\}\).
Then the forest  coincides  with the Hasse diagram of  \({\tilde \MS}\) viewed as a poset by the  inclusion relation:  the roots of the trees correspond to  the maximal elements of \(\MS\), and the orientation goes from the roots to the leaves, that are the vertices \(\{1\},\{2\},\ldots, \{n\} \) (see Figure \ref{labellednested3}).

\begin{figure}[h]
 
 \center
\includegraphics[scale=0.7]{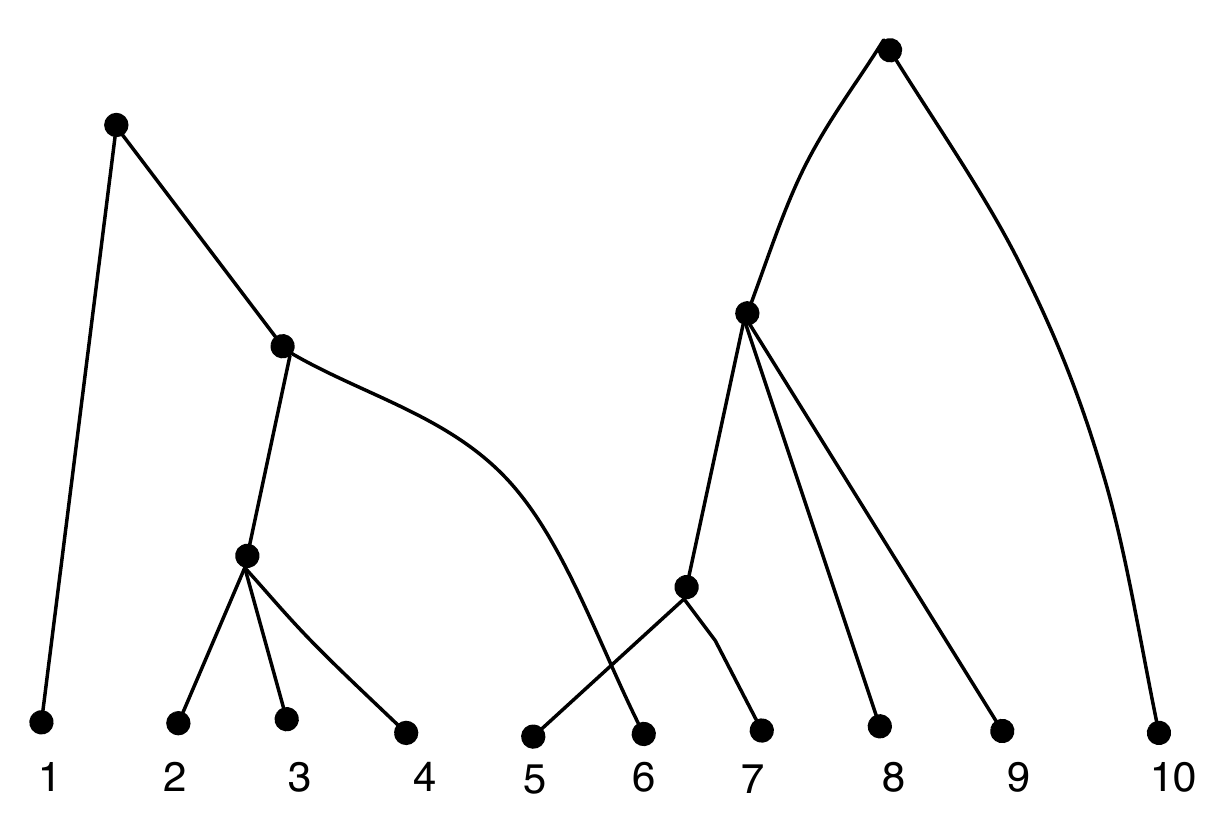}
\caption{This  forest  represents the  \(\Fc_{G(1,1,10)}\)-nested set
\(\MS\) made by the irreducibles  \( \{2,3,4\}\), \(\{2,3,4,6\}\), \(\{1,2,3,4,6\}\),
\(\{5,7\}\), \(\{5,7,8,9\}\), \(\{5,7,8,9,10\}\).}
\label{labellednested3}
\end{figure}

One can show, using a bijection proven in \cite{PeErdos}  (see \cite{gaifficayleynumbers} for a variant),   that there are  actions of `big' symmetric groups on the set of $\Fc_{G(1,1,n)}$-nested sets. Namely, the symmetric group \(S_{n+k-1}\) acts on the set of the $\Fc_{G(1,1,n)}$-nested sets \(\MS\) such that \(|\MS|=k\) and \(\MS\) includes \(\{1,2,...,n\}\).  These actions can be extended  to the basis of cohomology described in Theorem \ref{base coomologia}:  in \cite{callegarogaiffi3} (Theorem 10.1) it was shown that one can obtain, by counting the orbits of these actions,  an exponential (not recursive)  formula for a series that computes the Betti numbers of the models $Y_{G(1,1,n)}$.  

As we mentioned in the Introduction, some recursive  formulas for the Poincar\`e series of these varieties  are well known (see for instance  \cite{Manin}, \cite{YuzBasi}, \cite{GaiffiBlowups}). 

Let us describe  the formula obtained in  \cite{callegarogaiffi3}, that is the starting point  for our generalizations   in the next sections.  We denote by \(\Psi(q,t,z)\)  the following exponential generating series:
\[\Psi(q,t,z)= 1+ \sum_{n\geq 2, \; \MS}P(S) z^{|\MS|} \frac{t^{n+|\MS|-1}}{(n+|\MS|-1)!}\]
where, for every \(n\geq 2\), 
\begin{itemize}
\item  \(\MS\) ranges over all the nested sets of the building set \(\Fc_{G(1,1,n)}\);  
\item  \(P(\MS)\) is the polynomial, in the variable \(q\) (to be considered of degree 2),  that expresses the contribution to $H^*(Y_{G(1,1,n)}, \Z)$ provided by all the monomials \(m_f\) in the Yuzvinski basis such that \(supp \ f=\MS\).

\end{itemize}

\begin{teo}[see \cite{callegarogaiffi3}, Theorem 10.1]
 \label{teo:formulapoicareminimale}
 We have the following formula for the series \(\Psi(q,t,z)\): 
\[
 \Psi(q,t,z) = e^t\prod_{i\geq 3}e^{zq[i-2]_q\frac{t^i}{i!}}  
\]
 where \([j]_q\) denotes the \(q\)-analog of \(j\): \([j]_q=1+q+\cdots+q^{j-1}\).

 \end{teo}

\begin{rem}
In \cite{callegarogaiffi3} the  result above is stated about   the minimal projective De Concini-Procesi model, that is obtained starting from the projectivized hyperplane arrangement,  but, as it was shown in \cite{DCP1}, the integer homology rings of a  model \(Y_\G\) and of its corresponding projective model \({\overline Y}_\G\) are isomorphic.

\end{rem}

  \begin{es}

 In order to compute the Poincar\'e polynomial of \(Y_{G(1,1,6)}\) from the formula above one has to single out all the monomials in \(\Psi\) whose  \(z,t\) component is \(t^kz^s\) with \(k-s=5\).
 A product of the exponential functions that appear in the formula   gives:
 \[\frac{t^5}{5!} [1]+ \frac{t^6}{6!}z [42q+22q^2+7q^3+q^4]+ \frac{t^7}{7!}z^2 [35q^3+105q^2]\]
Therefore the  Poincar\'e polynomial  of \(Y_{G(1,1,6)}\), that is equal to the Poincar\'e polynomial of the moduli space \( {\overline M}_{0,7}\),  is \(1+(42q+22q^2+7q^3+q^4)+(35q^3+105q^2)=1+42q+127q^2+42q^3+q^4\).
 
 \end{es}

\section{Extension to  \(G(r,1,n)\), \(G(r,p,n)\)  and  \(G(r,r,n)\)}
\label{sec:extension}
In this section we will describe the building sets of irreducibles associated to the groups \(G(r,p,n)\), and their corresponding nested sets.

We notice that this  combinatorial setting could  be described also in the language of Dowling lattices: for instance,  the intersection poset of the reflection arrangement of type \(G(r,p,n)\) when \(p<r\) is the Dowling lattice \(Q_n(\Z_r)\) (see \cite{dowling}), and we are considering its  minimal building set and its  nested set complex, according to the combinatorial definition of Feichtner and Kozlov in \cite{feichtnerkozlovincidence}. For this combinatorial approach, and for its further extensions to Bergman geometry, one could refer to \cite{ardilaklivans},  \cite{ardilareinerwilliams}, \cite{delucchinested}, \cite{feichtnercomplexestrees}, \cite{hultman}.
In particular, in the first sections of \cite{delucchinested} (and in Remark 4.7) one can find an useful overview on the many ``bridges'' between combinatorics and geometry related to this subject. 

\subsection{The building set of irreducibles}

The reflecting hyperplanes of the arrangement in \(\C^n\) associated with the full monoidal group \(G(r,1,n)\) coincide with the reflecting hyperplanes of \(G(r,p,n)\) when \(p<r\) and have been described in the Introduction.
 
One can easily check that the building set of irreducibles \(\Fc_{G(r,1,n)}\) (that is equal to \(\Fc_{G(r,p,n)}\) when \(p<r\))  is made by two families of subspaces. The subspaces in the first family are   the  {\em strong} subspaces \(\overline{H}_{i_1,i_2,...,i_k}\) (the adjective  `strong' comes from the analysis of the \(B_n\) and \(D_n\) case in \cite{YuzBasi}), that are  the  annihilators of the subspaces in \(\C^n\) described by the equations
\[x_{i_1}=x_{i_2}=\cdots= x_{i_t}=0\]
where \(1\leq t \leq n\). We can represent them by associating to \(\overline{H}_{i_1,i_2,...,i_t}\) the subset \(\{0,i_1,i_2,...,i_t\}\) of \(\{0,1,...,n \}\).
The second family is made by   the  {\em weak} subspaces, that are  the  annihilators \(H_{i_1,i_2,...,i_t}(\alpha_2,...,\alpha_t)\) 
of the subspaces in \(\C^n\) described by the equations:
\[x_{i_1}=\zeta^{\alpha_2}x_{i_2}=\cdots= \zeta^{\alpha_t}x_{i_t}\]
where \(2\leq t \leq n\) and, for every \(s\), \(0\leq \alpha_s\leq r-1\).
\begin{rem}[Notation]
\label{notation rem}
Let us suppose that \(i_1<i_2<\cdots < i_t\); then we can represent these weak subspaces by associating to \(H_{i_1,i_2,...,i_t}(\alpha_2,...\alpha_t)\) the  {\em weighted} subset \(\{i_1, \stackrel{\alpha_2}i_2,...,\stackrel{\alpha_t}i_t\}\) of \(\{1,...,n\}\). The weights are integers modulo \(r\), and here (and in the sequel) if a weight is 0 we will omit to write it.

\end{rem}


We observe that the building set of irreducibles \(\Fc_{G(r,r,n)}\)  can be obtained from   \(\Fc_{G(r,1,n)}\) by removing  some   strong subspaces, namely the  hyperplanes \(\overline{H}_{i}\), for every \(i=1,...,n\).
Moreover, if \(r=2\), i.e. in the \(D_n\) case, one needs to remove also the two dimensional subspaces  \(\overline{H}_{i,j}\).
In fact we notice that the subspaces \(\overline{H}_{i,j}\) are irreducible only if \(r\geq 3\): since the \(r\) lines \(H_{i,j}(\alpha)\)  (with \(0\leq \alpha\leq r-1\))  belong to \(\Fc_{G(r,r,n)}\),  when \(r\geq 3\) it is not true that the two dimensional  subspace \(\overline{H}_{i,j}\) is the direct sum of the maximal subspaces of \(\Fc_{G(r,r,n)}\) contained in it.

\subsection{Nested sets for \(\Fc_{G(r,1,n)}\) and  \(\Fc_{G(r,r,n)}\)}

The  nested sets  \(\Sc=\{A_1,...,A_m\}\) for \(\Fc_{G(r,1,n)}\) or \(\Fc_{G(r,r,n)}\), according to Definition \ref{Gnested},  are characterized by the following properties: 
\begin{itemize}
\item given  any pair of subspaces \(A_i,A_j\in \Sc\), they are in direct sum or they are one included into the other;

\item if there are strong subspaces in \(\Sc\), they are linearly ordered by inclusion;

\item a strong subspace is never included into a weak subspace.

\end{itemize}
According to the representation of the irreducibles by  subsets of \(\{0, 1,....,n\}\) (more precisely, the subsets that do not  contain 0 are weighted subsets) described in the preceding section, a nested set is represented by a set \(\{A'_1,...,A'_m\}\) of (possibly weighted) subsets of \(\{0,...,n\}\) with the following properties:
\begin{itemize}
\item  the subsets that contain 0  are not weighted; they   are linearly ordered by inclusion; 
\item the subsets that do not  contain 0 are weighted; 
\item for any pair of subsets  \(A'_i,A'_j\), we have that, forgetting their weights, they are one included into the other or  disjoint; if  \(A'_i,A'_j\) both represent weak subspaces one  included into the other (say \(A'_i \subset A'_j\)), then their weights must be compatible. Since  we adopt for \(A'_i\) and \(A'_j\)  the  notation  of Remark \ref{notation rem}, this means that, up to the multiplication of all the weights of   \(A'_i\) by  the same power of \(\zeta\),  the weights associated to the same numbers must be equal.

\end{itemize}

We can represent a nested set by an oriented {\em weighted forest} as in Figure \ref{grafopesato}.
Every internal vertex \(v\) represents the   subset obtained considering the leaves that belong to the subtree stemming from \(v\). In particular, if \(v\) is weak, it represents a weighted subset, in the following way.
First of all,  if  a weak vertex is the root of a tree,  we  put its weight to be equal to 0 (so we don't write it). Then, given any weak vertex \(w\), the weight that one appends to the leaf \(i\) in its associated  weighted subset  is given by the sum (modulo \(r\))  of the weights that one finds in the oriented connected path from \(w\) to \(i\) (we do not take into account  the weight of \(w\)).
According to the above described rules, there is a unique way to put the  weights in the weighted forest respecting   the  notation  of Remark \ref{notation rem}.

 \begin{figure}[h]
 
 \center
\includegraphics[scale=0.6]{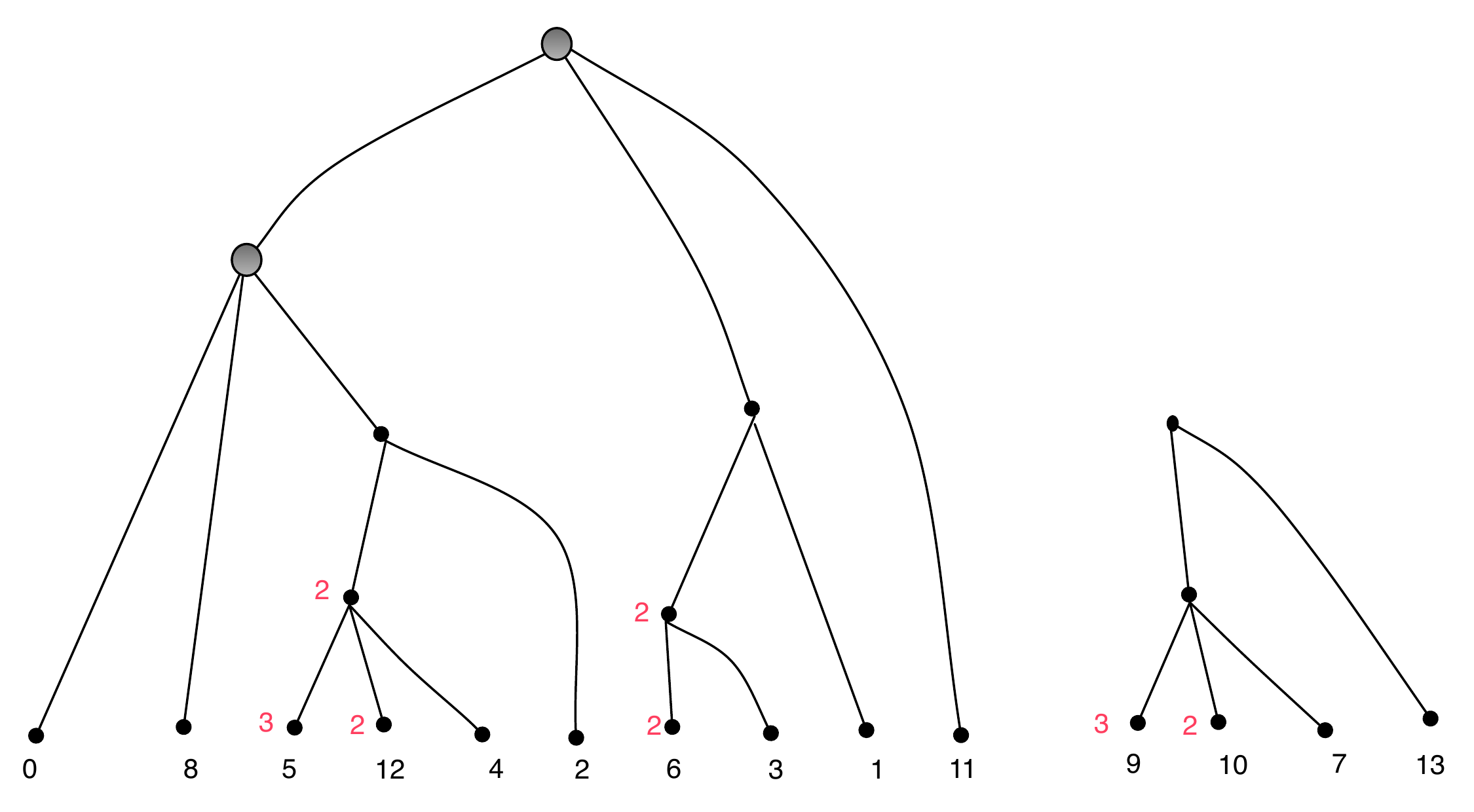}
\caption{A nested set in the case \(G(4,1,13)\). The 'big' internal vertices represent the strong subspaces. The red numbers are the weights. The nested set represented in the picture is therefore made by the  strong subspaces: \(\{0,2,4,5,8,12\}, \{0,1,2,3,4,5,6,8,11,12\}\) and by the weak subspaces \(\{\stackrel{} 4, \stackrel{\textcolor{red}{3}}5,\stackrel{\textcolor{red}{2}}{12}\}\), \(\{\stackrel{}2, \stackrel{\textcolor{red}{2}}4,\stackrel{\textcolor{red}{1}}5, \stackrel{}{12}\}\), \(\{\stackrel{}3, \stackrel{\textcolor{red}{2}}6\}\), \(\{\stackrel{}1, \stackrel{\textcolor{red}{2}}3,\stackrel{}6\}\), \(\{\stackrel{}7, \stackrel{\textcolor{red}{3}}9,\stackrel{\textcolor{red}{2}}{10}\}\), \(\{\stackrel{}7, \stackrel{\textcolor{red}{3}}9,\stackrel{\textcolor{red}{2}}{10}, \stackrel{}{13}\}\). }
\label{grafopesato}
\end{figure}

\section{Series}
\label{sec:series}
The results of Section \ref{classicbraid} on the computation of Betti numbers of \(Y_{G(1,1,n)}\) can be extended to the models \(Y_{G(r,1,n)}\) (\(=Y_{G(r,p,n)}\) when  \(p<r\))  and \(Y_{G(r,r,n)}\).
Our goal is to give a non recursive formula for the Poincar\'e series
\[ \Phi_{G(r,a)}(q,t)=1+ \sum_{n\geq 2} Poin(Y_{G(r,a,n)})(q)  \frac{t^{n}}{n!}\]
where  \(Poin(Y_{G(r,a,n)})(q) \) is the Poincar\'e polynomial of the model \(Y_{G(r,a,n)}\).

We start by singling out the contribution to \(\Phi_{G(r,a)}(q,t)\) given by the monomials of the Yuzvinski basis whose support does not contain strong subspaces. In terms of the Dowling lattice, the supports of these monomials are nested sets whose elements belong to the  subposet \(Q_n^0(\Z^r)\) of \(Q_n(\Z^r)\) defined by Hultman in \cite{hultman}.
\begin{defin}
Let \(a=1\) or \(a=r\) and let  us denote by \(K_{G(r,a)}(q,t,z)\)  the following exponential generating series:
\[K_{G(r,a)}(q,t,z)= 1+ \sum_{n\geq 2, \; \MS}P(S) z^{|\MS|} \frac{t^{n+|\MS|-1}}{(n+|\MS|-1)!}\]
where, for every \(n\geq 2\) (while \(r\) remains fixed),
\begin{itemize}
\item  \(\MS\) ranges over all the nested sets of the building set \(\Fc_{G(r,a,n)}\) that do not contain strong subspaces;  
\item  \(P(\MS)\) is the polynomial, in the variable \(q\),  that expresses the contribution to $H^*(Y_{\Fc_{G(r,a,n)}}, \Z)$ provided by all the monomials \(m_f\) in the Yuzvinski basis such that \(supp \ f=\MS\). 

\end{itemize}
\end{defin}

We notice that the series \(K_{G(r,a)}(q,t,z)\) doesn't change  in the two cases \(G(r,1,n)\) or \(G(r,r,n)\),  since only weak subspaces are involved. 

\begin{teo}
 \label{prop:seriesoloweak}
 We have the following formula for the series \(K_{G(r,a)}(q,t,z)\) (when \(a=1\) or \(a=r\)): 
\[
K_{G(r,a)}(q,t,z) = e^{t}\prod_{i\geq 3}e^{\frac{z}{r}q[i-2]_q\frac{(rt)^i}{i!}}  
\]
 where \([j]_q\) denotes the \(q\)-analog of \(j\): \([j]_q=1+q+\cdots+q^{j-1}\).

\end{teo}
 \begin{proof} This  is a  variant of Theorem \ref{teo:formulapoicareminimale},  nevertheless we write the details of the proof  for the convenience of the reader.  
 
A  monomial \(m\) of the Yuzvinsky basis of \(Y_{G(r,1,n)}\)  can be represented by a  weighted partition with exponents, in the following way.
The  support of the monomial is given, as we know from Theorem \ref{base coomologia}, by a \(\Fc_{G(r,1,n)}\)-nested set, and we are considering the monomials such that this nested set is made by weak subspaces. 

 For instance in  \(H^{*}(Y_{G(4,1,13)})\) the support of the monomial 
\[  c_{\{\stackrel{} 4, \stackrel{\textcolor{red}{3}}5,\stackrel{\textcolor{red}{2}}{12}\}} c^2_{\{\stackrel{}2, \stackrel{\textcolor{red}{2}}4,\stackrel{\textcolor{red}{1}}5, 8, 11, \stackrel{}{12}\}}c_{\{\stackrel{}7, \stackrel{\textcolor{red}{3}}9,\stackrel{\textcolor{red}{2}}{10}\}}c^3_{\{1,3,6, \stackrel{}7, \stackrel{\textcolor{red}{3}}9,\stackrel{\textcolor{red}{2}}{10}, \stackrel{}{13}\}}\]
is described  by the weighted forest in Figure \ref{grafopesato2}.

 \begin{figure}[h]
 
 \center
\includegraphics[scale=0.6]{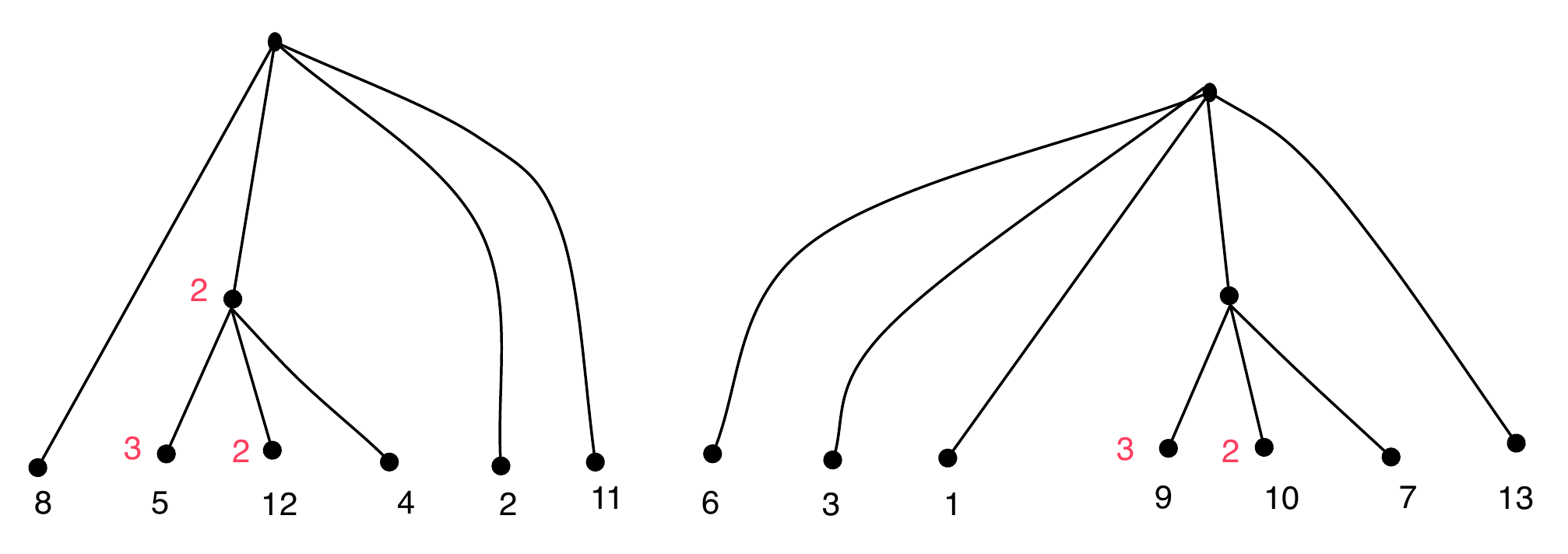}
\caption{This weighted forest represents the support of the monomial
\(c_{\{\stackrel{} 4, \stackrel{\textcolor{red}{3}}5,\stackrel{\textcolor{red}{2}}{12}\}} c^2_{\{\stackrel{}2, \stackrel{\textcolor{red}{2}}4,\stackrel{\textcolor{red}{1}}5, 8, 11, \stackrel{}{12}\}}c_{\{\stackrel{}7, \stackrel{\textcolor{red}{3}}9,\stackrel{\textcolor{red}{2}}{10}\}}c^3_{\{1,3,6, \stackrel{}7, \stackrel{\textcolor{red}{3}}9,\stackrel{\textcolor{red}{2}}{10}, \stackrel{}{13}\}}\) in \(H^{*}(Y_{G(4,1,13)})\).}
\label{grafopesato2}
\end{figure}

Then, following \cite{gaifficayleynumbers}, we can label the internal vertices of  the forest as in Figure \ref{grafopesato2b}:  we put labels level by level, and the label of a vertex   \(v\) is less than the label of a  vertex \(w\) iff  the subtree that stems from \(v\) contains a leaf whose label  is smaller than the labels of all the leaves in the subtree that stems from \(w\). If the forest has more than one connected component (as in Figure \ref{grafopesato2b}), we add an extra  vertex on top, with the maximum label. 

 \begin{figure}[h]
 
 \center
\includegraphics[scale=0.6]{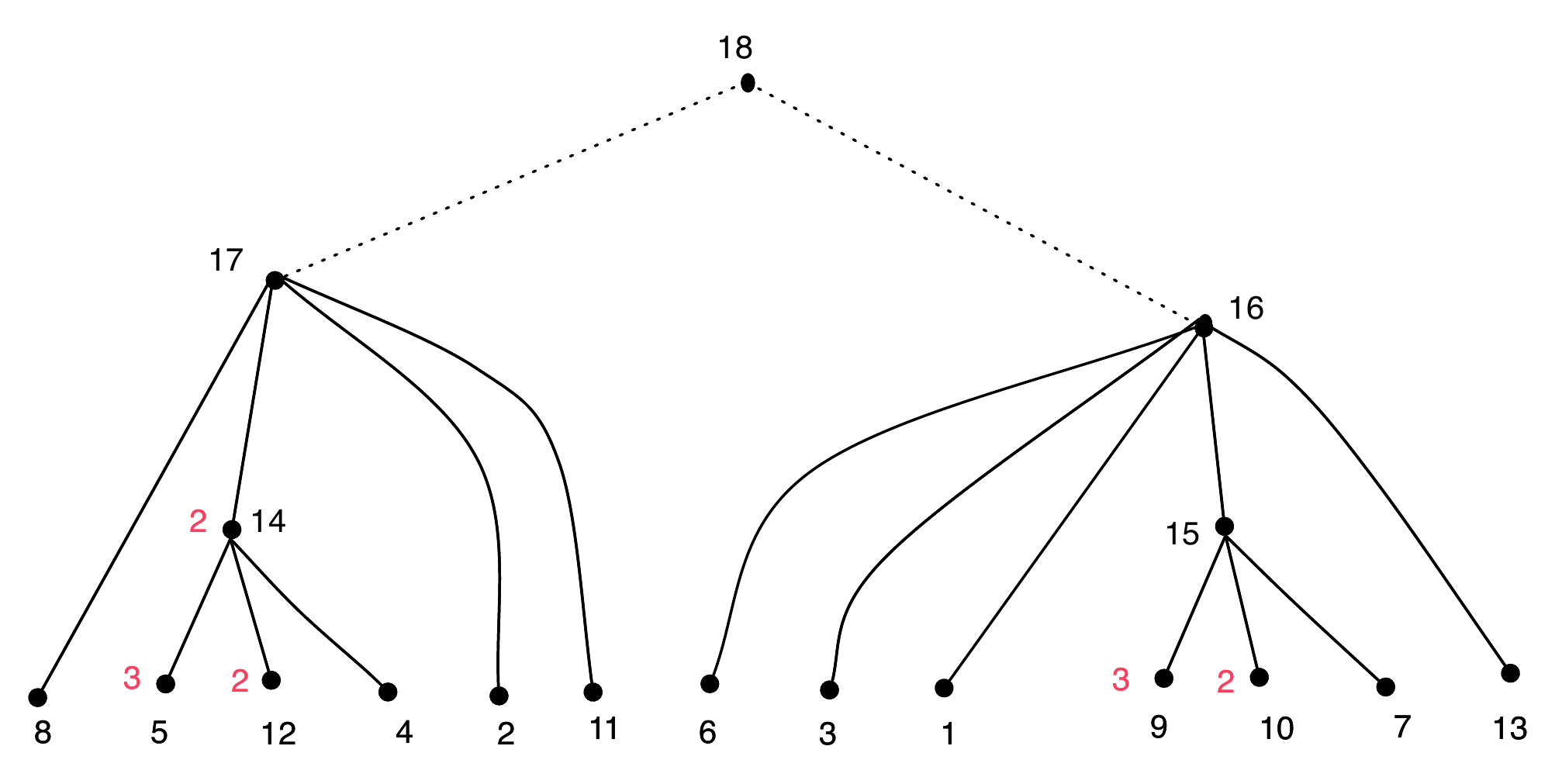}
\caption{We can label  (level by level) the internal vertices of the forest of  Figure \ref{grafopesato2} with the numbers \(14,15,16,17\). Since the forest is not connected, we add a vertex on top, with the maximum label \(18\).}
\label{grafopesato2b}
\end{figure}

We can then associate to the support of the monomial \(m\) a weighted partition, by looking at the internal vertices of the labelled forest and taking into account, for each internal vertex,  the labels and weights  of the vertices covered by it.
For instance, looking at the weighted forest represented in Figure \ref{grafopesato2b}, we associate to it the 
the weighted  partition: 
\[  \{\stackrel{} 4, \stackrel{\textcolor{red}{3}}5,\stackrel{\textcolor{red}{2}}{12}\} \{\stackrel{}2,  8, 11,\stackrel{\textcolor{red}{2}}{14}\}\{\stackrel{}7, \stackrel{\textcolor{red}{3}}9,\stackrel{\textcolor{red}{2}}{10}\}\{1,3,6,  \stackrel{}{13}, 15\}, \{16,17\}\]

Finally we can associate     to the monomial \[  c_{\{\stackrel{} 4, \stackrel{\textcolor{red}{3}}5,\stackrel{\textcolor{red}{2}}{12}\}} c^2_{\{\stackrel{}2, \stackrel{\textcolor{red}{2}}4,\stackrel{\textcolor{red}{1}}5, 8, 11, \stackrel{}{12}\}}c_{\{\stackrel{}7, \stackrel{\textcolor{red}{3}}9,\stackrel{\textcolor{red}{2}}{10}\}}c^3_{\{1,3,6, \stackrel{}7, \stackrel{\textcolor{red}{3}}9,\stackrel{\textcolor{red}{2}}{10}, \stackrel{}{13}\}}\]
the following weighted partition of \(\{1,...,17\}\) with exponents  attached to the parts, in order to   keep into account the exponents in the monomial: 
\[  \{\stackrel{} 4, \stackrel{\textcolor{red}{3}}5,\stackrel{\textcolor{red}{2}}{12}\} \{\stackrel{}2,  8, 11,\stackrel{\textcolor{red}{2}}{14}\}^2\{\stackrel{}7, \stackrel{\textcolor{red}{3}}9,\stackrel{\textcolor{red}{2}}{10}\}\{1,3,6,  \stackrel{}{13}, 15\}^3\{16,17\}^0\]

\begin{rem}
This process, if applied to a forest with more then one connected component, adds an extra vertex on top.  In this case we obtain a part with exponent 0, like \(\{16,17\}^0\) in the example above.
We notice that this part, by construction, contains 17, the maximum  of \(\{1,2,...,17\}\).
\end{rem}

Going back to our proof,   let us denote by \(\Bc(r,1)_{weak}\) the set of all the Yuzvinski basis monomials of all the models  \(Y_{G(r,1,n)}\) (\(n\geq 3\)) whose support is made by weak subsets.  As we remarked above, we can think of these monomials as weighted partitions with exponents  attached to the parts.  Let us now denote by \(\Pc(r,1)\) the union, for every \(j\geq 3\), of the set of weighted partitions of \(\{1,..,j\}\)  with exponents, such that:
\begin{itemize}
\item at  most one of the parts has exponent equal to 0 (and cardinality \(\geq 2\)). If this part exists, it contains the maximum number \(j\);

\item the other parts \(I\)  have cardinality  \(\geq 3\) and their exponent \(\alpha_I\) satisfies \(1\leq \alpha_I\leq |I|-2\).

\end{itemize} 
  As an easy corollary of Theorem 2.1 in \cite{gaifficayleynumbers}, we know that the above described map from \(\Bc(r,1)_{weak}\)  to  \(\Pc(r,1)\)  is bijective\footnote{We remark that another  bijection  between these two sets can be deduced from Theorem 1 of  \cite{PeErdos}.}, therefore we can find a formula for \(K_{G(r,a)}(q,t,z)\) by counting  the contribution of all the elements of  \(\Pc(r,1)\).

Then  we single out the contribution given to \(K_{G(r,a)}(q,t,z)\) by all the parts    represented by subsets with cardinality \(i\geq 3\) and with nonzero exponent. 
  If in a weighted partition there is only one such part  its contribution is \(\frac{z}{r}(q+q^2+\cdots+ q^{i-2})\frac{(rt)^i}{i!}\). If there are \(j\) such parts the associated  contribution is \((\frac{z}{r})^j(q+q^2+\cdots+ q^{i-2})^j \frac{(\frac{(rt)^i}{i!})^j}{j!}\). 
  In conclusion the contribution  of all the parts    represented by subsets with cardinality \(i\geq 3\) and with nonzero label is provided by
 \[e^{\frac{z}{r}(q+q^2+\cdots+ q^{i-2})\frac{(rt)^i}{i!}}-1\]

 Let us now focus on the contribution to \(K_{G(r,a)}(q,t,z)\) that comes  from the parts with cardinality \(i\geq 2\) and with exponent  equal to 0. For every monomial in the basis there is  at most one such  part (that  must contain the highest number of the set that we are partitioning), and its contribution is \(\frac{t^{i-1}}{(i-1)!}\). We note  that by construction all the weights of the numbers that belong to  this part are equal to 0, since it  does not represent a subspace in the support of the monomial.  

The total contribution of the elements with exponent  equal to 0 is therefore     \(\sum_{i\geq 2}\frac{t^{i-1}}{(i-1)!}\). Summing up, we observe that the expression  
  \[ e^t\prod_{i\geq 3}e^{\frac{z}{r}q[i-2]_q\frac{(rt)^i}{i!}}  \]
  allows us to take into account the contribution to \(K_{G(r,a)}(q,t,z)\) of all the  elements  in \(\Pc(r,1)\).
   \end{proof}

\subsection{Series for \(G=G(r,1,n)\)}
\label{section:seriesgr1n}
Let us now find a formula for the Poincar\'e series \(\Phi_{G(r,1)}(q,t)\) of the models  \(Y_{G(r,1,n)}\).   

In this case the  support of  a monomial of the Yuzvinsky basis can be represented by a forest that has a shape like the one suggested in Figure \ref{schemaforesta}.

 There may be at most one  connected component with strong vertices. Then the  strong  vertices, as in the picture, form a chain, and below each of them there is a subgraph made by weak vertices.  There may be other connected components  that are  made by weak vertices.

 \begin{figure}[h]
 
 \center
\includegraphics[scale=0.4]{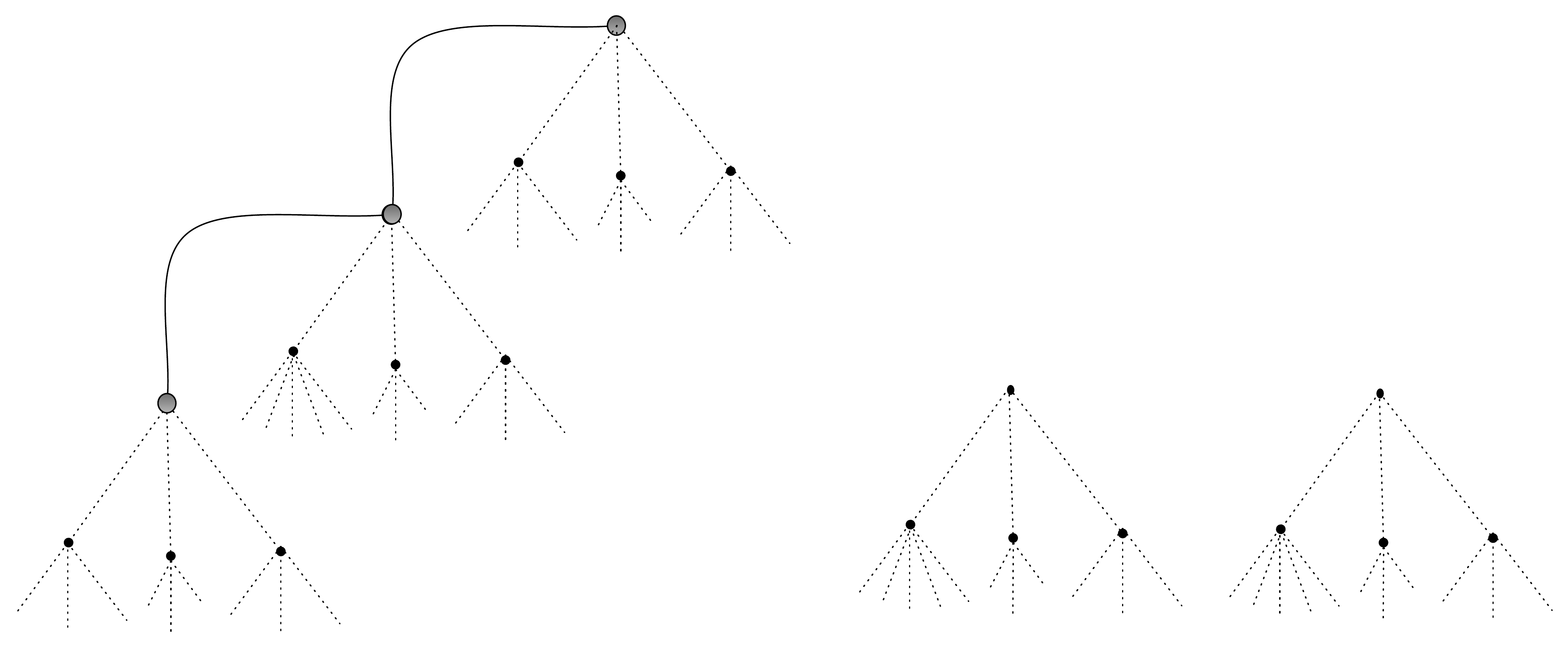}
\caption{The shape of a forest  that represents the support of a basis monomial of \(H^*(Y_{G(r,1,n)})\).   There is  at most one  component with strong vertices. Then the  strong  vertices, as in the picture, form a chain, and below each of them there is a subgraph made by weak vertices.  There may be other connected components  that are  made by weak vertices.}
\label{schemaforesta}
\end{figure}
 The following series will play a key role in computing   the contribution of a weak subtree that stems from a strong vertex:
\[\gamma_{G(r,1)}(q,t,z)=\left( \sum_{i\geq 2}\frac{t^{i-1}}{(i-1)!}q[i-1]_q\right) \prod_{i\geq 3}e^{q[i-2]_q\frac{z}{r}\left( \frac{tr}{i!}\right)^i}\]
The idea is that we can  evaluate the series  \( \gamma_{G(r,1)}(q,t,z) \) in \(z=\frac{\partial}{\partial t}\) and then  integrate  (formally, with constant equal to 0) with respect to the variable \(t\).  As a  result we get a new series in the variables \(q,t\) which we  denote by \( \Gamma_{G(r,1)}(q,t)\):

\[\Gamma_{G(r,1)}(q,t)=\int \gamma_{G(r,1)}(q,t,\frac{\partial}{\partial t})\]
\begin{rem}
\label{rem:ordinederivata}
Here and in the sequel, when we say that we   evaluate a  series in  \(z=\frac{\partial}{\partial t}\), we mean that first we compute the series, then, in  every monomial that appears in the final expression of the series, we put \(z=\frac{\partial}{\partial t}\).
\end{rem}
At the same way we define
\[{\mathcal K}_{G(r,1)}(q,t)= 1+ \int K_{G(r,1)}(q,t,\frac{\partial}{\partial t})\]
\begin{teo}
\label{teog(r1n)}
We have the following formula for the Poincar\'e series of the models \(Y_{G(r,1,n)}\):
\[\Phi_{G(r,1)}(q,t)=\frac{1}{1-\Gamma_{G(r,1)}(q,t)}{\mathcal K}_{G(r,1)}(q,t)\]

\end{teo}
\begin{proof}
First we observe that 
\[ {\mathcal K}_{G(r,1)}(q,t)= 1+ \int K_{G(r,1)}(q,t,\frac{\partial}{\partial t}) =1+t+ \sum_{n\geq 2, \; \MS \ weak}P(S)  \frac{t^{n}}{n!}=1+ t+\sum_{n\geq 2} Poin^w(Y_{G(r,1,n)})(q)  \frac{t^{n}}{n!}\]
where \(\MS \ weak\) means  that we are considering only the nested sets made by weak subspaces, and  \(Poin^w(Y_{G(r,1,n)})(q) \) is the contribution to the Poincar\'e  polynomial of the model \(Y_{G(r,1,n)}\) given by the basis monomials whose nested set is `weak'.

Then we observe that the difference between the series  \(\gamma_{G(r,1)}(q,t,z)\) and \(K_{G(r,1)}(q,t,z)\) consists only in the first exponential factor, where we find \(\sum_{i\geq 2}\frac{t^{i-1}}{(i-1)!}q[i-1]_q\) instead than \(\sum_{i\geq 2}\frac{t^{i-1}}{(i-1)!}\): the \(q\)-polynomial \(q[i-1]_q\) counts the contribution to the Poincar\`e series given by a strong vertex which covers \(i\) weak vertices in the graph.
The evaluation of \(z\) as \(\frac{\partial}{\partial t}\) and the integral transform \(\gamma_{G(r,1)}(q,t,z)\) into \(\Gamma_{G(r,1)}(q,t)\), that gives the correct contribution to the Poincar\'e series of a strong vertex and of the weak subgraph stemming from it.

Since the strong vertices are linearly ordered, if there are \(m\) strong vertices their contribution is given by  \(\Gamma_{G(r,1)}(q,t)^m\), so the  total contribution of strong vertices to \(\Phi_{G(r,1)}(q,t)\) is  
\[\Gamma_{G(r,1)}(q,t)+ \Gamma_{G(r,1)}(q,t)^2+\Gamma_{G(r,1)}(q,t)^3+...\]
Multiplying by \({\mathcal K}_{G(r,1)}(q,t)\) we take into account the contributions of the components of the forest that don't have  strong vertices, and this proves our  claim.

\end{proof}

\subsection{Series for \(G=G(r,r,n)\)}
\label{section:seriesgrrn}
\begin{teo}
\label{teog(rrn)}
When \(r\geq 3\), the  series \(\Phi_{G(r,r)}(q,t)\) coincides with  \(\Phi_{G(r,1)}(q,t)\).
When \(r=2\) (i.e. the \(D_n\) case),  we have

\[\Phi_{G(2,2)}(q,t)=\frac{(1-q\frac{t^2}{2})}{1-\Gamma_{G(2,1)}(q,t)}{\mathcal K}_{G(2,1)}(q,t)\]

\end{teo}
\begin{proof}
When \(r=2\), the only modification we have  with respect to the computation of Theorem \ref{teog(r1n)} is that among the forests that represent the supports of the  monomials we do not have forests whose lower strong vertex corresponds to a two dimensional subspace  \(\overline{H}_{i,j}\). The contribution to  \(\Phi_{G(2,1)}(q,t)\) of the associated  monomials  is computed by the series 
 \[\frac{q\frac{t^2}{2}}{1-\Gamma_{G(2,1)}(q,t)}{\mathcal K}_{G(2,1)}(q,t)\]
so we have to subtract it from \(\Phi_{G(2,1)}(q,t)\).
\end{proof}

\section{Real spherical models}
\label{sec:spherical}

When the group \(G\) is real we can construct   a {\em spherical model} associated to it, as a special case of a construction in   \cite{gaimrn}. 

Let us  consider the real arrangement \(\A_G\) given in the  euclidean space \(V\) by the reflecting hyperplanes of  \(G\), and let us   identify \(V^*\) with \(V\) by the scalar product.
Therefore the subspaces in the minimal building set, that we denote by  \(\Fc_{G}\),  are identified with subspaces in \(V\).
We will denote by \(\emme(G)\) the complement of the arrangement \(\A_G\) in \(V\).

Moreover, we denote  by $S(V)$ the  unit sphere in $V$, and, for 
every subspace $A\subset V$, let $S(A)=A\cap S(V)$. Let us consider 
the compact manifold 
\begin{displaymath}
	K=S(V)\times\prod _{A\in \Fc_{G}}S(A^{})
\end{displaymath}
There is an open embedding 
\(
\phi:\: 	\emme(G)/\R^+\longrightarrow K
\)
which  is obtained as a composition of the 
section $s:\:\emme(G)/\R^+ \mapsto \emme(G)$  
\begin{displaymath}
	s([p])=\frac{p}{|p|}\in S(V)\cap \emme(G)
\end{displaymath} 
with the map
\begin{displaymath}
 	\emme(G)\mapsto S(V)\times\prod _{A\in \Fc_{G}}S(A ^{})
\end{displaymath}
that is given by the normalization on the first factor and by the orthogonal projections followed by normalizations on the other factors.  
\begin{defin} 
We denote by $CY_{G}$ the closure in $K$ of $\phi(\emme(G)/\R^+)$.
\end{defin}
It turns out 
that 
$CY_{G}$ is a smooth manifold with corners that  has as many  connected components as the number of chambers of the arrangement.

A linear realization of $CY_{G}$ as a disjoint union of polytopes has been constructed in \cite{gaiffipermutonestoedra}. The polytopes involved lie inside  the chambers of the arrangement: in the fundamental chamber we find the graph associahedron  \(P_G\) of type \(G\) (i.e. the graph polytope defined in  \cite{devadoss} and \cite{carrdevadoss} associated with the Dynkin graph of type \(G\)) and in the other chambers there is the image of this polytope via the action of \(G\).
The faces of the graph  associahedron \(P_G\) that lies in the fundamental chamber are parametrized by  the {\em fundamental nested sets} defined as follows. Given a basis of simple roots for the root system of \(G\), the {\em fundamental building set of irreducibles} is the set of irreducible subspaces that are spanned by  simple roots. The nested sets associated to this building set are the fundamental nested sets.

Now, as it is well known, the chambers of the arrangement are in bijection with the elements of \(G\). It follows that  the faces of the graph associahedra that are the components of   $CY_{G}$  are parametrized by pairs \((w,\Sc')\) where \(w\in G\) and \(\Sc'\) is a fundamental nested set. 

From Section 5 of \cite{gaimrn} we know that there  is a natural projection map \({\mathcal P}\) from $CY_{G}$ to \({\overline Y}_{\Fc_{G}}\), the real compact wonderful model associated with the group \(G\). This model is constructed in \cite{DCP1} as the closure of the image of the map 
\[\Proj(\emme(G)) \longrightarrow \Proj(V)\times\prod _{A\in \Fc_{G}}\Proj(A^{})\]
where we identify  the euclidean space \(V\) with its dual as above.

Then Theorem 5.2 of \cite{gaimrn} states that \({\mathcal P}\)  is \(2^{k+1}\rightarrow 1\) when restricted to  the open parts of the boundary components of codimension \(k\) of $CY_{G}$. Therefore \({\overline Y}_{\Fc_{G}}\) can be obtained by glueing  \(|G|\) copies of  the polytope \(P_G\) in a prescribed way.

\section{Counting faces of polytopes in the real cases}
\label{Eulercomputation}

A variant of the computation of the Poincar\'e series of the model \(Y_G\) allows us to find  a series that counts the number of faces of \(CY_G\). From this one immediately obtains formulas for the number of faces of the graph associahedron \(P_G\).

\subsection{Case \(A_{n-1}\), group \(S_n=G(1,1,n)\).}
When \(G=S_n\), i.e. when we are dealing with the arrangement of type \(A_{n-1}\), the polytope \(P_G\) is a \(n-2\) dimensional Stasheff's associahedron.  The entries of the \(f-\)vectors of these associahedra are the well known Kirkman-Cayley numbers.  Anyway,  in order  to `test' our method, in this section we show how we can obtain a  series that encodes all the information about the \(f\)-vectors of the Stasheff's associahedra. 

First we notice that every  nested set in \(\Fc_{G(1,1,n)}\) is represented by a tree on \(n\) leaves, as we explained in Section \ref{classicbraid}. Let us now consider planar pictures of these trees, adding an ordering condition on the leaves: even if two trees represent the same nested set, we consider them as different objects if the order from left to right of their leaves is different. 
Each of these plane trees represents a face of  $CY_{G(1,1,n)}$: the plane trees where    the leaves,  from left to right, form the list  \(1,2,..,n\) represent the fundamental nested sets, so  the trees whose leaves form the list  \(w(1),w(2),...,w(n)\)  (where \(w\in S_n\))  describe the faces \((w,\Sc')\) of the polytope that  lies in the chamber corresponding to \(w\). 

In conclusion, the faces of the disjoint union of polytopes $CY_{G(1,1,n)}$ are in bijective correspondence with the set of the above described  {\em plane rooted  trees on \(n\) ordered leaves}.

\begin{defin}\label{conto}
Let us denote by \(F(z, t)\) the series

\begin{equation} 
F(z, t)= \sum_{n\geq 2} \sum_{T}z^{s(T)}\frac{t^{n+s(T)-1}}{(n+s(T)-1)!}
\end{equation}
where $T$ ranges over all the  above described   plane rooted trees  with $n$  ordered leaves and  $s(T)$ is the number of internal vertices in  $T$.
\end{defin}

We observe that we can write 
\[F(z, t)= \sum_{n\geq 2, 1\leq s\leq n-1} \frac{C_{n,s}}{(n+s-1)!}z^{s}t^{n+s-1}\]
where $C_{n,s}$ is the number of $(s-1)$-codimensional faces of $CY_{G(1,1,n)}$. Therefore,   the coefficient in \(F(z,t)\) of \(z^st^{n+s-1}\) multiplied by 
\(\frac{(n+s-1)!}{n!}\) is equal to the number of $(s-1)$-codimensional faces of the \(n-2\) dimensional Stasheff's associahedron, that is the Kirkman-Cayley number
\[D_{n+1,s-1}= \frac{1}{s}\binom{n-2}{s-1}\binom{n+s-1}{s-1}\]
\begin{prop}\label{serie}
We have \(F(z,t)=e^{z\frac{t^{2}}{1-t}}-1\). 
\end{prop}
\begin{proof}

As in the proof  of Theorem \ref{prop:seriesoloweak} we want to count partitions of \(\{1,...,n+k-1\}\)  into \(k\) parts with cardinality \(\geq 2\) instead than plane trees with \(n\) ordered leaves and \(k\) internal vertices.
An important difference is that here the order of the leaves  is relevant; this can be translated in terms of partitions by considering partitions where each part is internally ordered. In fact an ``ordered''  variant of the bijection described in Section \ref{sec:series} (see Theorem  2.2 of \cite{gaifficayleynumbers}) associates to  such a partition  a tree on \(n\) ordered leaves, as illustrated by Figure \ref{orderedtree}. 

\begin{figure}[h]
 
 \center
\includegraphics[scale=0.6]{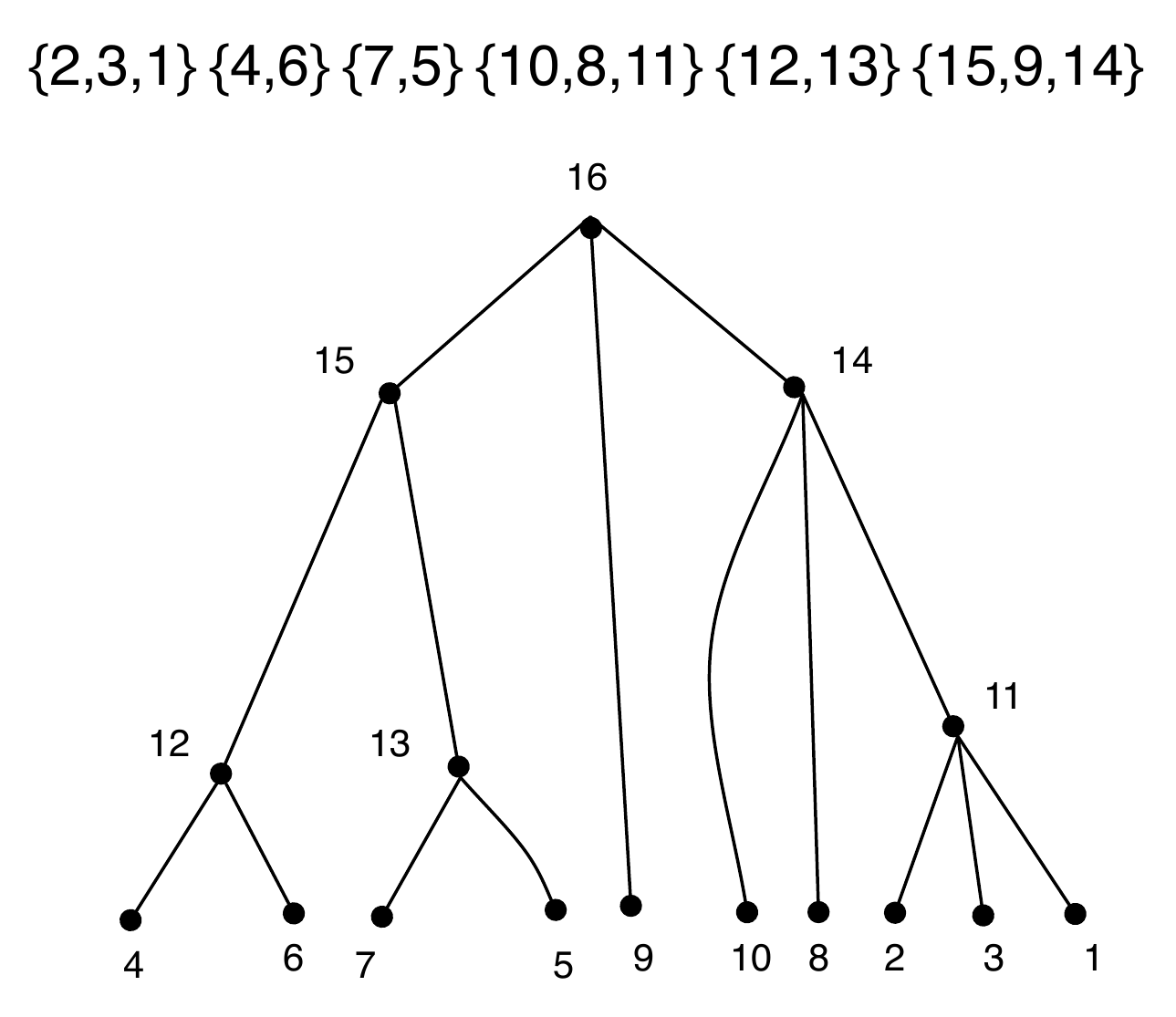}
\caption{The  internally ordered  partition of \([1,15]\) into 6 blocks  that is on top of the picture (the internal orderings  are obtained reading from left to right), produces the  rooted tree with 10 ordered leaves on the bottom of the picture. For every internal vertex \(v\), the labels of the vertices and leaves covered by it, read  from left to right, reproduce the internal orderings of a set of the partition. }
\label{orderedtree}
\end{figure}

 Furthermore, in this case all the parts have exponent equal to 1 (in particular there is no part with exponent 0), and the parts  have cardinality \(\geq 2\).

When  we single out the contribution given to \(F(z,t)\) by all the parts  represented by subsets with cardinality \(i\geq 2\)  we observe that 
  if in a  partition there is only one such part  its contribution is \(zt^i\) and if there are \(j\) such parts their  contribution is \(z^j \frac{(t^i)^j}{j!}\). 
  The formula follows.

\end{proof}
Let us now consider the series 
\[X(t)=\sum_{n\geq 2}\chi({\overline Y}_{G(1,1,n)})\frac{ t^{n-1}}{(n-1)!}\]
where \(\chi( \quad )\) denotes the Euler characteristic. There are several different ways to compute this series 
and a closed formula for \(\chi({\overline Y}_{\mathcal{F}_{G(1,1,n)}})\) has been provided in \cite{devadosstessellation} and \cite{hendersonrains}. Anyway we remark that we can compute \(X(t)\) also from \(F(z,t)\):

\begin{cor}
The series \(X(t)\) can be obtained from the series \(e^{\frac{z}{2}\frac{t^{2}}{1+t}}-1\) by evaluating \(z\) as \(\frac{\partial}{\partial t}\).
\end{cor}
\begin{proof}
We can compute the Euler characteristic by counting the faces of the \(CW\) complex that covers \({\overline Y}_{\mathcal{F}_{\mathcal{A}_{n-1}}}\) according to the map \({\mathcal P}\) described in the end of Section \ref{sec:spherical}.
Then  \(e^{\frac{z}{2}\frac{t^{2}}{1+t}}-1\) is \(F(\frac{z}{2},-t)\): the substitution \(z\rightarrow \frac{z}{2}\) in \(F(z,t)\) takes into account the \(2^{k +1}\rightarrow 1\) glueings of  the \(k\)-codimensional boundary components of $CY_{G(1,1,n)}$, while the substitution \(t\rightarrow -t\) gives the correct sign \((-1)^{n-1+s}(=(-1)^{n-1-s})\) to the \((s-1)\)-codimensional cells of the \(CW\) complex.

\end{proof}
\subsection{Cases \(B_n\), group \(G(2,1,n)\) and \(D_n\), group \(G(2,2,n)\).}
Let us consider \(a=1\) or \(a=2\) and define the series 
\[F_{CY_{G(2,a)}}(w,t)=\sum_n\sum_{1\leq j\leq n} C_{G(2,1),j}w^j\frac{t^n}{n!}\]

where $C_{G(2,a,n),j}$ is the number of $(j-1)$-codimensional faces of $CY_{G(2,a,n)}$.
In view of the description of fundamental nested sets, and of the parametrization by pairs \((w,\Sc')\) of the faces of $CY_{G(2,a,n)}$,
 this  number coincides with the number of  {\em plane rooted weighted trees on \(n+1\)  leaves labelled by \(\{0,1,...,n\}\) and ordered from left to right} with the following properties (see Figure \ref{grafopesato222}):
\begin{itemize}

\item the vertices may be weak or strong; the root is a strong vertex, and  the strong vertices are linearly ordered;
\item the leaves belong to a horizontal  line; the first leaf on the left is the leaf 0, that is contained in the subgraph that stems from every strong vertex; the other leaves are put on the line in any order from left to right;

\item there are at least two edges that go down  from  every internal vertex; in the \(D_n\) case, i.e. if \(a=2\),  there must be at least four  edges that go down from  the lower strong vertex. 

\item there are weights 0 or 1 attached to the leaves and vertices, with the following restrictions: the root has weight 0 and when a vertex covers some other objects (where with  `object' we mean a vertex or a leaf), the leftmost object has weight 0. Furthermore, let us say that a leaf has parity 0 or 1 if the sum (modulo 2) of the weights that one finds in the path from the leaf to the root is  0 or 1 respectively; then, in the \(D_n\) case   the number of leaves with parity 1 must be even.
\item in the \(D_n\) case, the first leaf on the right of the leaf 0  may be  equipped with and  extra label (\(+\), \(-\) or \(\pm\)). More precisely, it has this  extra label if it is covered by a weak vertex,  with the following meaning: let \(c\) be the label of the leaf on the right of \(b\), then if in the subspace represented by the weak vertex there is the root \(x_c+x_b\) we put the label \(+\), if there is the root \(x_c-x_b\) we put the label \(-\). We also allow the following special case: there is a weak vertex  that covers exactly two leaves, namely the leaf \(b\) with the extra label \(\pm\)  and the leaf \(c\);  this notation means that in the nested set there are both the one dimensional subspaces spanned by the roots  \(x_c+x_b\) and \(x_c-x_b\).
\end{itemize}

 \begin{figure}[h]
 
 \center
\includegraphics[scale=0.6]{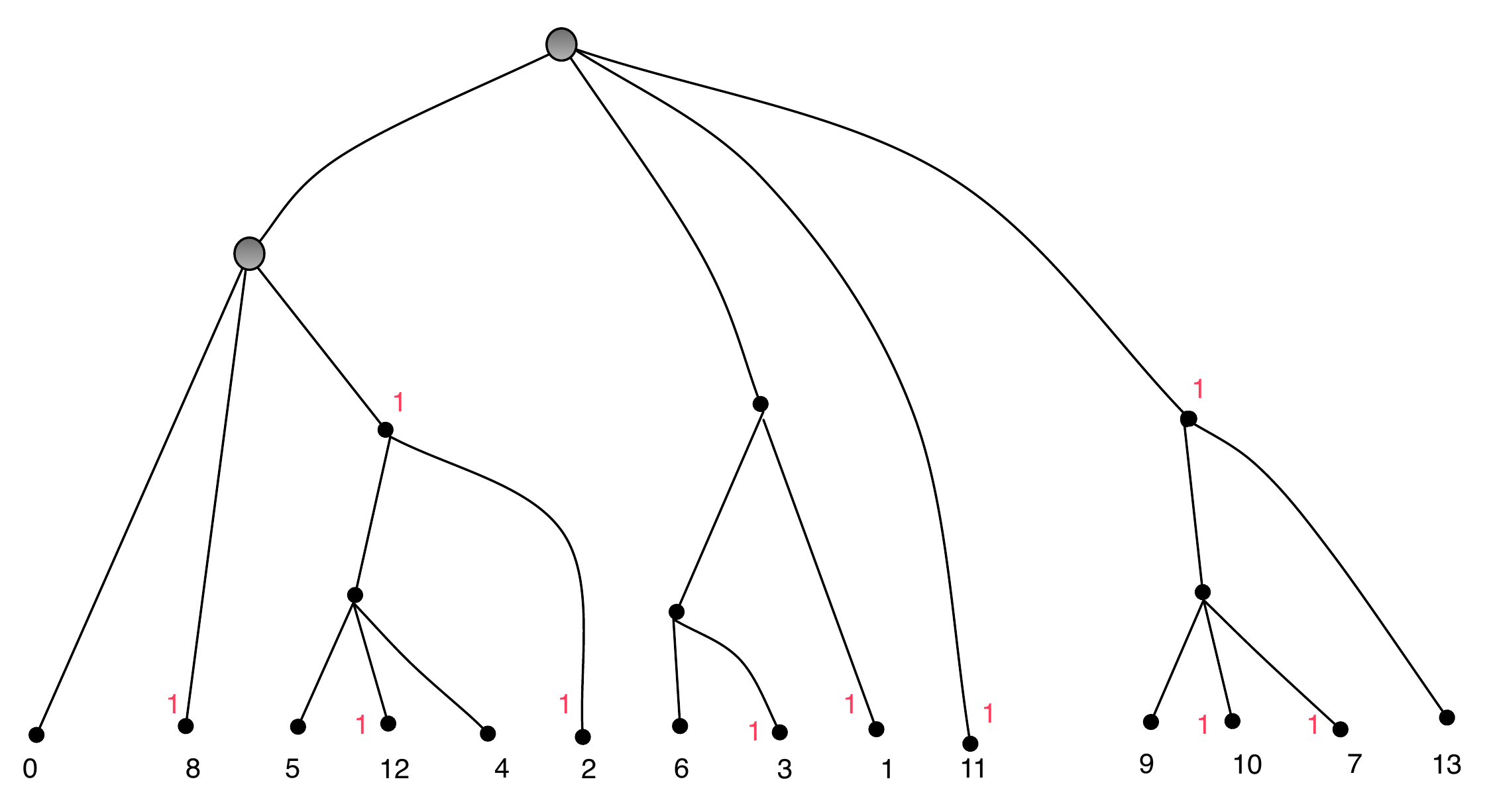}
\caption{A valid plane oriented rooted tree that represents a face of  codimension 7 of \(CY_{G(2,1,13)}\).}
\label{grafopesato222}
\end{figure}

\begin{rem}We notice that the combined information provided by the ordering from left to right of the leaves  and by the weights  of these trees  determines the chamber of the arrangement that we are considering, since the chambers are in bijective correspondence with the orderings of the numbers \(\{1,2,...,n\}\),
with weights \(0\) or \(1\) attached to these numbers (these weights are recovered from the weights of a tree by attaching to a  leaf its parity  0 or 1).
\end{rem}

When \(a=1\), i.e. in the \(B_n\) case,  the associated polytope is again the   \((n-1)\)-dimensional Stasheff's associahedron whose \(f\)-vector is well known.  Nevertheless we show a formula for   \(F_{CY_{G(2,1)}}(w,t)\) from which we will   obtain  a formula for \(F_{CY_{G(2,2)}}(w,t)\).  We remark that since   \(CY_{G(2,2,n)}\), for  every \(n\), is the disjoint union of   \(2^{n-1}n!\) polytopes, this immediately gives a  formula  that computes the \(f\)-vectors of the graph associahedra of type  \(D_n\).

One defines 
\[{\widetilde \gamma}_{G(2,1)}(t,z,w)=2\left( \sum_{i\geq 1}i(2t)^{i-1}\right) \prod_{j\geq 2}e^{\frac{wz}{2}\left( 2t\right)^j}\]

This series plays the same role of the series \(\gamma_{G(r,1)}\) in Section \ref{section:seriesgr1n}: it counts the contribution to  \(F_{CY_{G(2,1)}}(w,t)\) of the  vertices of the weak subgraph  that stems from  a strong vertex (the leaves are considered weak except for the leaf 0). In this case, for a technical reason that will become clear later,  there are two variables, \(w\) and \(z\), that count the same quantity, i.e. the number of internal vertices of this subgraph.
As in the \(G(1,1,n)\) case, the order from left to right of the leaves is relevant and therefore we are considering partitions into internally ordered parts. This time there are also  weights, equal to 0 or 1,  attached to each number of the partition.

The factor \(2\left( \sum_{i\geq 1}i(2t)^{i-1}\right)\) has the following meaning: in a given partition there is a part, say of cardinality \(i\), that contains the highest number, and represents the vertices connected by an arc to the strong vertex; then the contribution of this part  is given by \(2^iit^{i-1}\), where the factor \(i\) counts   the choices of the position of the highest number.

To obtain the exact contribution of the  vertices of the weak subgraph  that lies below a strong vertex, we have to pass from \(\gamma_{G(r,1)}\) to the following series:
\[{\widetilde \Gamma}_{G(2,1)}(t,w)=\int {\widetilde \gamma}_{G(2,1)}(t,\frac{\partial}{\partial t},w)\]
where  the substitution of \(z\) with \(\frac{\partial}{\partial t}\) is performed as specified in Remark \ref{rem:ordinederivata}.
We observe that the variable \(w\) still counts the internal vertices of the subgraph.
\begin{teo}We have 
\[F_{CY_{G(2,1)}}(w,t)= \frac{1}{1-w{\widetilde \Gamma}_{G(2,1)}(t,w)}\]
\end{teo}
\begin{proof}
We have already discussed the meaning of the series \({\widetilde \Gamma}_{G(2,1)}(t,w)\). 
So, since the strong vertices are linearly ordered, the series \(F_{CY_{G(2,1)}}(w,t)\) is equal to 
\[w{\widetilde \Gamma}_{G(2,1)}(t,w)+(w{\widetilde \Gamma}_{G(2,1)}(t,w))^2+(w{\widetilde \Gamma}_{G(2,1)}(t,w))^3+...\]
\end{proof}

\begin{es}
\label{es:numeriBn}
A simple computation shows that the first terms of the series \(w{\widetilde \Gamma}_{G(2,1)}(t,w)\) are as follows:
\[w{\widetilde \Gamma}_{G(2,1)}(t,w)=w2t+(w^2+w)2^2t^2+(2w^3+3w^2+w)2^3t^3+(5w^4+10w^3+6w^2+w)2^4t^4+...\]
As a consequence, the first terms of \(F_{CY_{G(2,1)}}(w,t)\) are:
\[F_{CY_{G(2,1)}}(w,t)=1+w2t+(2w^2+w)2^2t^2+(5w^3+5w^2+w)2^3t^3+ (14 w^4+21 w^3+9w^2+w)2^4t^4+...\]
The coefficients that appear inside the parentheses   give, as expected, the  \(f\)-vectors of Stasheff's associahedra.
\end{es}

We can now quickly  obtain a formula in the \(D_n\) case
\begin{teo}We have 
\label{teo:politopodn}
\[F_{CY_{G(2,2)}}(w,t)= \frac{(1-t)w{\widetilde \Gamma}_{G(2,1)}(t,w)-2tw-2t^2w-2t^2w^2}{1-w{\widetilde \Gamma}_{G(2,1)}(t,w)}\]
In particular, when \(n\geq 4\), the coefficient of \(w^st^n\) of this series, divided by \(2^{n-1}\), gives the number of faces of codimension \(s-1\) of the  \((n-1)\)-dimensional associahedron of type \(D\). 
\end{teo}
\begin{proof}
A first difference with respect to the computation in the \(B_n\) case is that we have to remove all  the graphs such that the subgraph that stems from the lower strong vertex   has  only two or three leaves, namely the leaf 0 plus  one or two leaves. A subgraph  of this type that has  only two leaves   gives the contribution \(2wt\) to \({\widetilde \Gamma}_{G(2,1)}(t,w)\), while if it has  three   leaves it gives the contribution \(4(w+w^2)t^2\) to \({\widetilde \Gamma}_{G(2,1)}(t,w)\). 
Furthermore, we have to take into account the extra label \(+\), \(-\) or \(\pm\) of  the leaf on the right of the leaf \(0\) (we will call it the leaf \(b\)): this extra label appears only when the leaf \(b\) is covered by a weak vertex. So in order to compute the contributions of the extra labels \(+\) and \(-\) we  multiply by two the contribution of lower strong vertices computed so far:
\[2[w{\widetilde \Gamma}_{G(2,1)}(t,w)-2tw-2^2t^2(w+w^2)]\]
and then subtract from it the contribution of lower strong vertices where the leaf \(b\) is not covered by a weak vertex:
 \[ 2[w{\widetilde \Gamma}_{G(2,1)}(t,w)-2tw-2^2t^2(w+w^2)]-2t[w{\widetilde \Gamma}_{G(2,1)}(t,w)-2wt]\] 
Then we  add \(2^2t^2w^2\) to take into account the case where the extra label \(\pm\) appears.
Finally, we obtain the following formula that computes the contribution of lower strong vertices:
\[ \frac{1}{2}\left(2[w{\widetilde \Gamma}_{G(2,1)}(t,w)-2tw-2^2t^2(w+w^2)]-2t[w{\widetilde \Gamma}_{G(2,1)}(t,w)-2wt]+2^2t^2w^2 \right )\] 
  The factor \( \frac{1}{2}\) takes into account that in the \(D_n\) case   we are considering only the trees where an even number of leaves have parity 1.
  
The contribution of strong vertices different from the lower one is computed dividing the formula that we have obtained by \(1-w{\widetilde \Gamma}_{G(2,1)}(t,w) \). 
\end{proof}
\begin{es}
If one computes \(F_{CY_{G(2,2)}}(w,t)\) from the formula above, starting from the formula for the first terms of \(w{\widetilde \Gamma}_{G(2,1)}(t,w)\) shown in  Example \ref{es:numeriBn}, one obtains that the coefficient of \(2^3t^4\) in \(F_{CY_{G(2,2)}}(w,t)\) is equal to: 
\[  16w^4+24w^3+10w^2+w \]
The coefficients that appear above give the  \(f\) vector of the graph associahedron of type \(D_4\). We notice that a  formula for the number of the vertices  of these graph associahedra (in terms of the Catalan numbers) is provided by  Proposition 5.4 of \cite{postnikov}.

We notice that the coefficient of \(2^2t^3\) in \(F_{CY_{G(2,2)}}(w,t)\) is \(5w^3+5w^2+w\), giving the \(f\)-vector of the \(2\)-dimensional Stasheff's associahedron. This reflects the fact that the `degenerate' root system \(D_3\) is equal to \(A_3\).  
\end{es}
As a corollary of the results above, we can describe a series that computes the Euler characteristic of the real compact models of type \(B_n\) and \(D_n\) (closed formulas can be found in \cite{hendersonrains}):
 \begin{cor}
For \(a=1\) or \(a=2\), if  we evaluate the series \(F_{CY_{G(2,a)}}(w,-t)\) in \(w=-\frac{1}{2}\) we obtain the series 
\[X_{CY_{G(2,a)}}(w,t)=\sum_n\sum_{1\leq j\leq n} \chi({\overline Y}_{G(2,a,n)})\frac{t^n}{n!}\] 
 that computes the Euler characteristic of the real compact models \({\overline Y}_{G(2,a,n)}\).
 
 \end{cor}
\begin{proof}
This is an immediate consequence of the properties of the map \({\mathcal P}\), in particular of the  \(2^{k +1}\rightarrow 1\) glueings  of the \(k\)-codimensional boundary components of $CY_{G(2,a)}$.
\end{proof}

\addcontentsline{toc}{section}{References}
\bibliographystyle{acm}
\bibliography{Bibliogpre} 
\end{document}